\documentclass[a4paper,oneside,12pt,notitlepage]{article}
\usepackage{amssymb}
\usepackage{amsmath}
\usepackage{amsthm}
\usepackage{multirow}
\usepackage{fullpage}
\usepackage{graphicx}
\usepackage[cp1250]{inputenc}
\usepackage[IL2]{fontenc}

\newtheorem{Theorem}{Theorem}

\newtheorem{Corollary}[Theorem]{Corollary}

\newtheorem{Conjecture}[Theorem]{Conjecture}

\theoremstyle{definition}

\begin{document}
\title{Elementary proof of Rayleigh formula for graphs  \\[20pt]
SVO\v{C} 2007  \\[10pt]}
\author{Josef Cibulka\thanks{Department of Applied Mathematics, Charles University, Malostransk\'e n\'am.~25, 118~00~Praha~1, Czech Republic. Email: \texttt{cibulka@kam.mff.cuni.cz}}
 \and
       Jan Hladk\'y  \thanks{Department of Applied Mathematics, Charles University, Malostransk\'e n\'am.~25, 118~00~Praha~1, Czech Republic. Email: \texttt{hladky@kam.mff.cuni.cz}}}
\date{}
\maketitle
\begin{abstract}
The Rayleigh monotonicity is a principle from the theory of electrical networks.
Its combinatorial interpretation says for each two edges of a graph $G$, that 
the presence of one of them in a random spanning tree of $G$ is negatively correlated 
with the presence of the other edge.
In this paper
we give a self-contained (inductive) proof of Rayleigh monotonicity for graphs.
\end{abstract}

Rayleigh monotonicity refers to an intuitive principle in the
theory of electrical networks: the total resistance between two nodes in
the network does not decrease when we increase the resistance of
one branch. 

We will refer to a \emph{graph} for what is sometimes called a multigraph
in literature, i.e., two vertices may be connected with several edges.
When we speak about a subgraph of a graph, we refer only to its edge set;
the subgraph is always spanning.

The network can be viewed as a graph whose vertices are nodes 
and edges are the branches of the network. The graph is weighted, each 
edge has weight equal to the reciprocal of the resistance of the respective branch.

%Rayleigh monotonicity BLABLABLA:. Together with Kirchhoff's formula
%\cite{Kirch} it can be used to show an intuitively obvious  fact from the
%theory of electrical networks: the total resistance between two nodes in
%an electrical network does not decrease when we increase the resistance of
%one branch.
%
Let $G=(V,E,w)$ be a graph with weighted edges where $w:E\rightarrow
\mathbb{R}^+$ is its weight function. For $I\subseteq E$ we define the
weight of $I$ by $w(I)=\prod_{e\in I}w(e)$ and for a family $\mathcal{F}$
of sets of edges, $\mathcal{F}\subseteq 2^E$, we define its weight
$\|\mathcal{F}\|=\sum_{I\in \mathcal{F}}w(I)$.

We will use $\mathcal{T}_{e_1,e_2}$ to denote the family of spanning trees of $G$ that 
contain edges $e_1$ and $e_2$. Similarly,
$\mathcal{T}_{e_1,\overline{e_2}}$, $\mathcal{T}_{\overline{e_1},e_2}$ and
$\mathcal{T}_{\overline{e_1},\overline{e_2}}$ denote the families of
spanning trees containing the edges without a bar and not containing the
edges with a bar.

In 1847, Kirchhoff showed \cite{Kirch} that the resistance between the end-vertices of an edge $e_1$
of the network is equal to
$
\frac{1}{w(e_1)}\|\mathcal{T}_{e_1}\| / \|\mathcal{T}\| ,
$
where $\mathcal{T}$ is the family of all spanning trees of $G$ and $\mathcal{T}_{e_1}$ 
is the family of spanning trees containing $e_1$.
%Contracting the end-vertices of a branch $e_2$ has the same effect on the resistence in 
%the network as setting the resistence of $e_2$ to zero. 
Rayleigh monotonicity principle implies that contracting an edge $e_2$ 
does not increase the resistance between the end-vertices of $e_1$. Therefore
\[
\frac{\|\mathcal{T}_{e_1}\|}{\|\mathcal{T}\|} \geq \frac{\|\mathcal{T}_{e_1,e_2}\|}{\|\mathcal{T}_{e_2}\|},
\]
%which is equivalent to the following theorem. 
which is equivalent to Theorem \ref{ThmRayleightMonotonity}. Recently, Youngbin Choe \cite{Choe} found a 
combinatorial proof of the theorem; the proof uses Jacobi Identity and All Minors Matrix-Tree Theorem.
In this paper we give a self-contained combinatorial proof.

% $\mathcal{T}_{\overline{e_1},\overline{e_2}}$ denotes the family of
% spanning trees containing $e_1$ and not $e_2$, $e_2$ and not $e_1$ and not
% $e_1$ nor $e_2$, respectively.

\begin{Theorem}\label{ThmRayleightMonotonity}
\begin{equation}\label{eqRayleight}
\|\mathcal{T}_{e_1,\overline{e_2}}\| \|\mathcal{T}_{\overline{e_1},e_2}\|
\geq \|\mathcal{T}_{e_1,e_2}\|
\|\mathcal{T}_{\overline{e_1},\overline{e_2}}\|
\end{equation}
\end{Theorem}

\begin{proof}
Fix an orientation of $e_1$ and $e_2$. A subforest $F$ of $G$ is
\emph{important} if both $F\cup e_1$ and $F\cup e_2$ form a spanning tree
of $G$. Let $C$ be the unique cycle in $F \cup e_1 \cup e_2$. The cycle
$C$ contains both $e_1$ and $e_2$. We say that $F$ has \emph{positive
orientation} if $e_1$ and $e_2$ are consistently oriented along $C$.
Otherwise we say that $F$ has \emph{negative orientation}. Let
$\mathcal{C}_+$ and $\mathcal{C}_-$ be the set of all important forests
that have positive and negative orientation, respectively.

The statement will be proven by showing that
$$\|\mathcal{T}_{e_1,\overline{e_2}}\| \|\mathcal{T}_{\overline{e_1},e_2}\| -
\|\mathcal{T}_{e_1,e_2}\| \|\mathcal{T}_{\overline{e_1},\overline{e_2}}\|=
w(e_1)w(e_2)\left(\|\mathcal{C}_+\|-\|\mathcal{C}_-\|\right)^2,$$ or
equivalently,
\begin{equation}\label{eq_hlavni}
\|\mathcal{T}_{e_1,\overline{e_2}}\| \|\mathcal{T}_{\overline{e_1},e_2}\|
+ 2w(e_1)w(e_2)\|\mathcal{C}_+\|\|\mathcal{C}_-\|=
\|\mathcal{T}_{e_1,e_2}\| \|\mathcal{T}_{\overline{e_1},\overline{e_2}}\|+
w(e_1)w(e_2)\left(\|\mathcal{C}_+\|^2+\|\mathcal{C}_-\|^2\right).
\end{equation}

Equation (\ref{eq_hlavni}) can be viewed as an equality of two polynomials
in variables $w(e)$, $e\in E$. In order to prove it, we shall check that
the coefficient of every monomial is the same on the both sides. The
\emph{multiplicity} of edge $e$ in monomial $c \prod_{f\in E}
w(f)^{\alpha_f}$ is the number $\alpha_e$. An edge $e$ is \emph{present}
in monomial $c \prod_{f} w(f)^{\alpha_f}$ if its multiplicity is at least
one. An edge $e$ is \emph{plentiful} in monomial $c \prod_{f}
w(f)^{\alpha_f}$ if its multiplicity is at least two. The \emph{degree}
$d(v)$ of a vertex $v\in V$ is defined by $d(v)=\sum_{f \textrm{ incident
to } v} \alpha_f$. For every monomial $M=\prod_f w(f)^{\alpha_f}$ that is
contained (with nonzero coefficient) on one side of (\ref{eq_hlavni}), it
holds $\alpha_{e_1}=\alpha_{e_2}=1$, $\sum_f \alpha_f=2(|V|-1)$,
$0\leq\alpha_f\leq 2$ for every $f\in E$ and $d(v)>1$ for every $v\in V$.
Let $M$ be any such a monomial. We have to check that
\begin{equation}\label{eq_monomial}
A_{e_1:e_2}+2A_{+-}=A_{e_1 e_2:\emptyset}+A_{++}+A_{--},
\end{equation}
where
\begin{eqnarray*}
A_{e_1:e_2}&=& \# \: \{(T_1,T_2)\: | \: T_1 \in \mathcal{T}_{e_1,\overline{e_2}}, T_2\in \mathcal{T}_{\overline{e_1},e_2}, w(T_1)w(T_2)=M \} ,\\
A_{e_1 e_2:\emptyset}&=&\# \:  \{(T_1,T_2)\: | \: T_1 \in \mathcal{T}_{e_1,e_2}, T_2\in \mathcal{T}_{\overline{e_1},\overline{e_2}}, w(T_1)w(T_2)=M \},\\
A_{+-}&=&\# \:  \{(F_1,F_2)\: | \: F_1 \in \mathcal{C}_+, F_2\in \mathcal{C}_-, w(e_1)w(e_2)w(F_1)w(F_2)=M \},\\
A_{++}&=&\# \:  \{(F_1,F_2)\: | \: F_1, F_2\in \mathcal{C}_+, w(e_1)w(e_2)w(F_1)w(F_2)=M \},\\
A_{--}&=&\# \:  \{(F_1,F_2)\: | \: F_1, F_2\in \mathcal{C}_-,
w(e_1)w(e_2)w(F_1)w(F_2)=M \}.
\end{eqnarray*}

We will prove Equation (\ref{eq_monomial}) by induction on the number of
vertices of $G$. First, we should check that Equation (\ref{eq_monomial})
holds for all graphs $G$ with at most 3 vertices. This can be easily done.
(Note, that there are infinitely many graphs with at most 3 vertices since
multiedges are allowed. This is not a problem as we can without loss of
generality assume that $G$ contains only edges present in $M$.)

Assume that $|V|=n>3$ and Equation (\ref{eq_monomial}) holds for every
weighted graph $G'=(V',E',w')$, $|V'|<n$, any choice of edges $e'_1, e'_2
\in E'$ and any feasible monomial $M'$. Every time we use the induction
hypothesis, our graph $G'$ will live on a proper subset of vertices of the
graph $G$; edges $e'_1, e'_2$ will be the same as in the induction step,
i.e.,\ $e'_1=e_1, e'_2=e_2$, unless stated otherwise.

We may assume that the graph $G$ is loopless; we leave out the loops because they do
not change any of the terms in (\ref{eq_monomial}). Since $\sum_f
\alpha_f=2(|V|-1)$, there is a vertex $v$ such that $d(v)\leq 3$.
Moreover, we can choose $v$ such, that $d(v) \in \{2,3\}$ and if $d(v) = 3$ then $v$ is
incident to at most one of $e_1$ and $e_2$. We distinguish two cases.
% Moreover, we can choose $v$ such, that either $d(v)=2$ or $d(v)=3$ and at
% least two of the edges counted in the degree of $v$ are distinct from
% $e_1$ and $e_2$. We distinguish two cases.
\begin{enumerate}
\item $d(v)=2$.\\
Then either $v$ is incident to two edges (let us call them $h,i$) present in
$M$ or to one plentiful edge $h$ (then we set $i=h$). Recall that if $h$
is plentiful then $h \not = e_1, e_2$.
  \begin{enumerate}
  \item The edges $e_1,e_2$ do not coincide with $h,i$.\\
  Let $M'=M/(w(h)w(i))$, $G'=G-v$. From the induction hypothesis we know that 
  $A'_{e_1:e_2}+2A'_{+-}=A'_{e_1 e_2:\emptyset}+A'_{++}+A'_{--}$. 
  We can add arbitrary one of $h$ and $i$ to every spanning tree of $G'$ and every
  spanning tree of $G$ has at least one of $h$ and $i$. 
  Therefore $A_{e_1:e_2}=2A'_{e_1:e_2}$, $A_{e_1 e_2:\emptyset}=2A'_{e_1 e_2:\emptyset}$, 
  $A_{e_1:e_2}=2A'_{e_1:e_2}$, $A_{+-}=2A'_{+-}$, $A_{++}=2A'_{++}$, $A_{--}=2A'_{--}$ 
  and the statement follows.
  \item One of the edges $e_1,e_2$ coincides with $h,i$.\\
  Without loss of generality, let $e_1=h$. 
  Every important forest contains the edge $i$, so if $F_1$ and $F_2$ are important forests, 
  then $w(e_1)w(e_2)w(F_1)w(F_2)\not =M$. This implies that $A_{+-}=A_{++}=A_{--}=0$. The mapping
$$(T_1,T_2)\mapsto (T_1\triangle\{h,i\},T_2\triangle\{h,i\})$$
 is a bijection between partitions counted in $A_{e_1:e_2}$ and in $A_{e_1 e_2:\emptyset}$ and thus  Equation (\ref{eq_monomial}) holds.
  \item Edges $e_1$ and $e_2$ are exactly $h$ and $i$.\\
  Depending on the orientation of $e_1$ and $e_2$, one of the sets $\mathcal{C}_+$, $\mathcal{C}_-$ is empty. Assume, that $\mathcal{C}_-=\emptyset$. Then we have $A_{+-}=A_{--}=0$. There cannot exist a partition $(T_1,T_2)$ that would be counted in $A_{e_1 e_2:\emptyset}$; the edge set of $T_2$ would not span the vertex $v$. Thus $A_{e_1 e_2:\emptyset}=0$. The mapping
$$(F_1,F_2)\mapsto (F_1\cup e_1,F_2 \cup e_2)$$
is a bijection of partitions counted in $A_{++}$ and $A_{e_1:e_2}$. This
proves the statement.
  \end{enumerate}

\item $d(v)=3$.\\
Then either $v$ is incident to three edges $i,j,h$ present in $M$ or to
one edge $j$ and one plentiful edge $h$ (then we set $i=h$).
  \begin{enumerate}
  \item None of the edges $i,j,h$ coincides with $e_1,e_2$.\\
  We can write $A_{e_1 :e_2}=A_{e_1 : e_2}^{i:j,h}+A_{e_1 : e_2}^{j:i,h}+A_{e_1 : e_2}^{h:i,j}$, where $A_{e_1:e_2}^{X:Y}$ is defined for two edge sets $X$ and $Y$ as
\begin{align*}
A_{e_1:e_2}^{X:Y}
&=& \# \: \{(T_1,T_2)\: | \: T_1 \in \mathcal{T}_{e_1,\overline{e_2}}, T_2\in \mathcal{T}_{\overline{e_1},e_2}, X\subseteq T_1, Y\subseteq T_2, w(T_1)w(T_2)=M \}+\\
& &\# \: \{(T_1,T_2)\: | \: T_1 \in \mathcal{T}_{e_1,\overline{e_2}},
T_2\in \mathcal{T}_{\overline{e_1},e_2}, Y\subseteq T_1, X\subseteq T_2,
w(T_1)w(T_2)=M \}.
\end{align*}

  For the numbers $A_{e_1 e_2:\emptyset}, A_{+-}, A_{++}$ and $A_{--}$ we define $A^{X:Y}_{e_1 e_2:\emptyset}, A^{X:Y}_{+-}, A^{X:Y}_{++}$ and $A^{X:Y}_{--}$ in a~similar fashion. We shall show that
  \begin{equation} \label{eq_spojovani}
  A_{e_1:e_2}^{X:Y}+2A_{+-}^{X:Y}=A_{e_1 e_2:\emptyset}^{X:Y}+A_{++}^{X:Y}+A_{--}^{X:Y}
  \end{equation}
  for $X=\{i\}$, $Y=\{j,h\}$. Then, by symmetry, analogous equalities for $X=\{j\}$, $Y=\{i,h\}$ and $X=\{h\}$, $Y=\{i,j\}$ also hold. Summing them up, we get the statement.

  The end-vertices of $i$ and $j$ different from $v$ will be denoted by $x$ and $y$, respectively. Let $M'=\frac{w(k)}{w(i)w(j)w(h)}M$, $G'=G-v+k$, where $k=xy$ is a new edge connecting vertices $x$ and $y$ ($xy$ may be a multiedge now). We have from the induction hypothesis $A'_{e_1:e_2}+2A'_{+-}=A'_{e_1 e_2:\emptyset}+A'_{++}+A'_{--}$. It is easy to see that $A'_{e_1:e_2}=A_{e_1:e_2}^{i:j,h}$, $A'_{e_1 e_2:\emptyset}=A_{e_1 e_2:\emptyset}^{i:j,h}$, $A'_{+-}=A_{+-}^{i:j,h}$, $A'_{+-}=A_{++}^{i:j,h}$, $A'_{--}=A_{+-}^{i:j,h}$ and thus (\ref{eq_spojovani}) holds.

  \item One of the edges $e_1,e_2$ coincides with $i,j,h$.\\
  Without loss of generality, assume that $h=e_1$. Let $h=vu, i=vx, j=vy, e_2=ab$ (with
  orientation $\overrightarrow{e_2}=\overrightarrow{ba}$). Let $G'=G-v+k_1$,
  $M'=\frac{w(k_1)}{w(i)w(j)w(k)}M$, $e'_1=k_1$, $e'_2=e_2$, where $k_1=xy$,
  $G''=G-v+k_2$, $M''=\frac{w(k_2)}{w(i)w(j)w(k)}M$, $e''_1=k_2$, $e''_2=e_2$, where
  $k_2=ux$; $G'''=G-v+k_3$, $M'''=\frac{w(k_3)}{w(i)w(j)w(k)}M$, $e'''_1=k_3$,
  $e'''_2=e_2$, where $k_1=uy$. We will use induction hypothesis for polynomials $M'$,
  $M''$ and $M'''$. For the edges $e'_1$, $e''_1$ and $e'''_1$ fix orientations
  $\overrightarrow{e'_1}=\overrightarrow{xy}, \overrightarrow{e''_1}=\overrightarrow{xu},
  \overrightarrow{e'''_1}=\overrightarrow{yu}$. Refix\footnote{Refixing the orientation
  will not change the validity of the Equation \ref{eq_hlavni}, since $A_{+-}=\overline{A}_{+-},
  A_{++}=\overline{A}_{--}, A_{--}=\overline{A}_++, A_{e_1:e_2}=\overline{A}_{e_1:e_2},
  A_{e_1 e_2:\emptyset}=\overline{A}_{e_1 e_2:\emptyset}$, where the barred variables
  correspond to the situation where orientation of one edge was changed.} orientation
  of $e_1$, $\overrightarrow{e_1}=\overrightarrow{vu}$.
   Then
  \begin{align}
  A_{e_1:e_2}&=A''_{e''_1:e''_2}+A'''_{e'''_1:e'''_2}+A'_{e'_1 e'_2:\emptyset}, \label{eqSoucet1}\\
  A_{e_1 e_2:\emptyset}&=A''_{e''_1 e''_2:\emptyset}+A'''_{e'''_1 e'''_2:\emptyset}+A'_{e'_1:e'_2}. \label{eqSoucet2}
  \end{align}

  We shall prove combinatorially that
  \begin{align}\label{eqSoucet3}
   A_{++}&+A_{--}-2A_{+-} = \\
   &= A''_{++}+A''_{--}-2A''_{+-} \: + \: A'''_{++}+A'''_{--}-2A'''_{+-} \: - \: A'_{++}-A'_{--}+2A'_{+-}.\notag
  \end{align}
  In order to do so, we will view $2A_{+-}$ as
  $$2A_{+-}=\# \:  \{(F_1,F_2)\: | \: F_l \in \mathcal{C}_+, F_{3-l}\in \mathcal{C}_-, w(e_1)w(e_2)w(F_1)w(F_2)=M, l\in \{1,2\} \}$$
  (and similarly we treat with $2A'_{+-}$, $2A''_{+-}$ and $2A'''_{+-}$). Let
  $(F_1^\diamondsuit,F_2^\diamondsuit)$ be any partition that is counted in
  $A'_{++}$, $A'_{--}$, $2A'_{+-}$, \ldots, $2A'''_{+-}$. Each of $F_1^\diamondsuit$
  and $F_2^\diamondsuit$ is a spanning forest of $G-v$ such, that adding the edge
  $e_2$ creates a spanning tree of $G-v$. Vertices $a$ and $b$ must be contained in
  distinct components of $F^\diamondsuit_l$ ($l=1,2$). Moreover, no component can
  contain all the three vertices $x,y,u$.

  Take any partition $(F_1,F_2)$ that is counted in $A_{++}, A_{--}$ or $2A_{+-}$ and
  delete from it the edges $i$ and $j$, $F^\heartsuit_l=F_l-\{i,j\}$. It is immediate to
  see that $(F^\heartsuit_1,F^\heartsuit_2)$ meets the conditions described for
  $(F^\diamondsuit_1,F^\diamondsuit_2)$ also. The \emph{trace} of a partition $(E_1,E_2)$
  (which is counted in one of $A'_{++}$, $A'_{--}$, $2A'_{+-}$, \ldots, $2A'''_{+-}$ or
  $A_{++}, A_{--}, 2A_{+-}$) is defined as $\{C_1\cap\{x,y,u\},C_2\cap\{x,y,u\}\}$, where
  $C_l$ is the vertex set of a component of $E_l$ containing the vertex
  $a$. Table~\ref{table_contributions} shows contribution of any
  kind of partition to the numbers appearing in
  Equation~(\ref{eqSoucet3}).
  Equation~(\ref{eqSoucet3}) holds since
  the contributions of partitions of
  any kind are the same to the left-hand side as to the right-hand side.
  \begin{table}[ht]
  \caption{Contributions of partitions of different traces to the Equation~(\ref{eqSoucet3}).}
  \centering
  \begin{tabular}{|lll|}
  \hline \hline
  trace&left-hand side&right-hand side\\ \hline
  \multirow{2}{*}{$\{\{u\},\{u\}\}$}& \multirow{2}{*}{$\Delta A_{--}=2$}&$\Delta A''_{--}=1$\\
  & &$\Delta A'''_{--}=1$\\ \hline
  $\{\{u\},\{y\}\}$&$\Delta 2A_{+-}=1$&$\Delta 2A'''_{+-}=1$\\ \hline
  $\{\{u\},\{y,u\}\}$&$\Delta A_{--}=1$&$\Delta A''_{--}=1$\\ \hline
  $\{\{u\},\{x\}\}$&$\Delta 2A_{+-}=1$&$\Delta 2A''_{+-}=1$\\ \hline
  $\{\{u\},\{x,u\}\}$&$\Delta A_{--}=1$&$\Delta A'''_{--}=1$\\ \hline
  \multirow{2}{*}{$\{\{u\},\{x,y\}\}$}& \multirow{2}{*}{$\Delta 2A_{+-}=2$}&$\Delta 2A''_{+-}=1$\\
  & &$\Delta A'''_{+-}=1$\\ \hline
  \multirow{2}{*}{$\{\{y\},\{y\}\}$}&\multirow{2}{*}&$\Delta A'_{--}=1$\\
  &&$\Delta A'''_{++}=1$\\ \hline
  $\{\{y\},\{y,u\}\}$&$\Delta 2A_{+-}=1$&$\Delta A'_{--}=1$\\ \hline
  $\{\{y\},\{x\}\}$&$\Delta A_{++}=1$&$\Delta 2A'_{+-}=1$\\ \hline
  \multirow{2}{*}{$\{\{y\},\{x,u\}\}$}&\multirow{2}{*}&$\Delta 2A'_{+-}=1$\\
  &&$\Delta 2A'''_{+-}=1$\\ \hline
  $\{\{y\},\{x,y\}\}$&$\Delta A_{++}=1$&$\Delta A'''_{++}=1$\\ \hline
  \multirow{2}{*}{$\{\{y,u\},\{y,u\}\}$}& \multirow{2}{*}&$\Delta A'_{--}=1$\\
  &&$\Delta A''_{--}=1$\\ \hline
  \multirow{2}{*}{$\{\{y,u\},\{x\}\}$}& \multirow{2}{*}&$\Delta 2A'_{+-}=1$\\
  &&$\Delta 2A''_{+-}=1$\\ \hline
  $\{\{y,u\},\{x,u\}\}$&$\Delta A_{--}=1$&$\Delta 2A'_{+-}=1$\\ \hline
  $\{\{y,u\},\{x,y\}\}$&$\Delta 2A_{+-}=1$&$\Delta 2A''_{+-}=1$\\ \hline
  \multirow{2}{*}{$\{\{x\},\{x\}\}$}& \multirow{2}{*}&$\Delta A'_{++}=1$\\
  &&$\Delta A''_{++}=1$\\ \hline
  $\{\{x\},\{x,u\}\}$&$\Delta 2A_{+-}=1$&$\Delta A'_{++}=1$\\ \hline
  $\{\{x\},\{x,y\}\}$&$\Delta A_{++}=1$&$\Delta A''_{++}=1$\\ \hline
  \multirow{2}{*}{$\{\{x,u\},\{x,u\}\}$}& \multirow{2}{*}&$\Delta A'_{++}=1$\\
  &&$\Delta A'''_{--}=1$\\ \hline
  $\{\{x,u\},\{x,y\}\}$&$\Delta 2A_{+-}=1$&$\Delta 2A'''_{+-}=1$\\ \hline
  \multirow{2}{*}{$\{\{x,y\},\{x,y\}\}$}& \multirow{2}{*}{$\Delta A_{++}=2$}&$\Delta A''_{++}=1$\\
  &&$\Delta A'''_{++}=1$\\ \hline
  \hline
  \end{tabular}\label{table_contributions}
  \end{table}
 
  From (\ref{eqSoucet1}), (\ref{eqSoucet2}) and (\ref{eqSoucet3}) we have
  \begin{align*}
  A_{++}+A_{--}-2A_{+-}+A_{e_1 e_2:\emptyset}-A_{e_1:e_2}&=\\
  =&A''_{++}+A''_{--}-2A''_{+-}+A''_{e''_1 e''_2:\emptyset}-A''_{e''_1:e''_2}+\\
   &+A'''_{++}+A'''_{--}-2A'''_{+-}+A'''_{e'''_1 e'''_2:\emptyset} -A'''_{e'''_1:e'''_2}-\\
   &-A'_{++}-A'_{--}+2A'_{+-}-A'_{e'_1 e'_2:\emptyset}+A'_{e'_1:e'_2}=\\
  =&0
  \end{align*}
  which was to be proven.
  \end{enumerate}
\end{enumerate}
\end{proof}

Theorem~\ref{ThmRayleightMonotonity} can be reformulated as a correlation inequality for
spanning trees in a graph. Let $\mathcal{P}$ be the probability distribution of the
spanning trees in graph $G$ proportional to their weights, $\mathcal{T}$ the set of all
the spanning trees. We have
$$\mathrm{Pr}_{T\sim\mathcal{P}}[T=T_0]=\frac{\|T_0\|}{\|\mathcal{T}\|}$$
for any fixed spanning tree $T_0$.
\begin{Corollary} Let $G$ be a connected graph. For any edges $e_1$ and $e_2$, such that $e_2$ is not a bridge we have
$$\mathrm{Pr}_{T\sim\mathcal{P}}[e_1\in T \:|\: e_2\not \in T]
\geq\mathrm{Pr}_{T\sim\mathcal{P}}[e_1\in T \:|\: e_2\in T].$$
\end{Corollary}

Let us note that a similar correlation inequality looks plausible if the spanning trees
are replaced by spanning forests. This conjecture was stated by Grimmett and Winkler in 
\cite{GrimmettWinkler} and 
is still open.
\begin{Conjecture}
Set $\mathcal{F}$ to be the set of all spanning forests of a (weighted) graph $G$,
$\mathcal{B}$ the probability distribution of the spanning forests where the probability
of each spanning forest is proportional to its weight. Let $e_1$ and $e_2$ be two
distinct edges of $G$. Then
$$\mathrm{Pr}_{F\sim\mathcal{B}}[e_1\in F \:|\: e_2\not \in F]
\geq\mathrm{Pr}_{F\sim\mathcal{B}}[e_1\in F \:|\: e_2\in F].$$
\end{Conjecture}

\bigskip
 The notion of the sets $\mathcal{T}_{e_1,e_2}$,
$\mathcal{T}_{e_1,\overline{e_2}}$, $\mathcal{T}_{\overline{e_1},e_2}$ and
$\mathcal{T}_{\overline{e_1},\overline{e_2}}$ can be naturally extended to matroids. For
a matroid $\mathcal{M}=(E,I)$ with weight $w:E\rightarrow \mathbb{R}^+$ on its elements
we define
\begin{equation*}
\mathcal{T}_{e_1,\overline{e_2}}=\{T\:|\:T\in I, r(T)=r(\mathcal{M}),
e_1\in T,e_2\not\in T\}
\end{equation*}
and $\mathcal{T}_{e_1,e_2}$, $\mathcal{T}_{\overline{e_1},e_2}$ and
$\mathcal{T}_{\overline{e_1},\overline{e_2}}$ similarly. (The two definitions are
consistent for graphic matroids.) A matroid is called a \emph{Rayleigh matroid} if it
satisfies Equation~\ref{eqRayleight} for any choice of distinct elements $e_1,e_2\in E$.
Graphic matroids are a proper subclass of Rayleigh matroids. See
\cite{ChoeWagner,Wagner05} for more details.

\section*{Acknowledgement and final remarks}
The research was conducted while the authors were supported by the DIMACS-DIMATIA REU grant (NSF CNS 0138973).

The authors recently learned that Michael LaCroix and David Wagner found independently a very
similar proof.


\begin{thebibliography}{99}
  \bibitem{Choe} Y. Choe,
   A combinatorial proof of Rayleigh formula for graphs,
   {\it preprint},
   \texttt{http://com2mac.postech.ac.kr/papers/2004/04-30.pdf}.

  \bibitem{ChoeWagner} Y.B. Choe and D.G. Wagner,
  Rayleigh Matroids,
  {\it Combin. Probab. Comput.},
  15(2006), 765-781.

	\bibitem{GrimmettWinkler} G. Grimmett and S. Winkler,
  Negative association in uniform forests and connected graphs
  {\it Random Structures and Algorithms}, 
  24 (2004) 444-460.

   \bibitem{Kirch} G. Kirchhoff,
   \"Uber die Aufl\"osung der Gleichungen, auf welche man bei der Untersuchungen der linearen Vertheilung galvanischer Str\"ome gef\"uhrt wird,
   {\it Ann. Phys. Chem.},
   72 (1847), 497-508.

  \bibitem{Wagner05} D. Wagner,
  Matroid inequalities from electrical network theory,
  {\it Electron. J. Combin.},
  11(2) (2005), A1.

\end{thebibliography}
\end{document}